\documentclass[12pt,a4paper]{article}

\usepackage[T1]{fontenc}       
\usepackage[utf8]{inputenc}    
\usepackage{microtype}         
\usepackage[margin=2.5cm]{geometry}  

\usepackage{mathtools}        
\usepackage{amssymb}          
\usepackage{amsthm}            
\usepackage{bm}               
\usepackage{mathrsfs}          
\usepackage{dsfont}            

\usepackage{graphicx}          
\usepackage{xcolor}          
\usepackage{float}    
\usepackage[normalem]{ulem}

\usepackage{tikz}
\usetikzlibrary{arrows,calc,positioning,external}

\usepackage[colorlinks, linkcolor=black, citecolor=black, urlcolor=black]{hyperref}
\usepackage[nameinlink,noabbrev]{cleveref}

\usepackage{caption}
\usepackage{subcaption}
\usepackage{wrapfig}
\usepackage[shortlabels]{enumitem}

\usepackage{authblk}
\usepackage[textsize=footnotesize]{todonotes}
\usepackage[backend=biber, style=alphabetic, maxbibnames=40, doi=false, isbn=false, url=false, eprint=false]{biblatex}
\numberwithin{equation}{section}

\theoremstyle{plain}
\newtheorem{theorem}{Theorem}[section]

\newtheorem{lemma}[theorem]{Lemma}
\newtheorem{corollary}[theorem]{Corollary}

\theoremstyle{definition}

\theoremstyle{remark}
\newtheorem{remark}[theorem]{Remark}

\newcommand{\numberset}{\mathbb}

\newcommand{\R}{\numberset{R}}

\DeclarePairedDelimiter{\abs}{\lvert}{\rvert}

\DeclarePairedDelimiter{\ceil}{\lceil}{\rceil}
\DeclarePairedDelimiter{\tond}{(}{)} 
\DeclarePairedDelimiter{\quadr}{[}{]}

\DeclareMathOperator{\supp}{Supp}
\DeclareMathOperator{\law}{law}

\DeclareMathOperator{\wmix}{w_{mix}}

\newcommand{\tmix}{\mathrm{t}_{\mathrm{mix}}}

\newcommand{\eps}{\varepsilon}

\newcommand{\tmixa}[1]{\tmix\tond*{#1}}

\newcommand{\hess}{\nabla^2}

\newcommand{\variwr}[2]{\mathrm{Var}_{#1}\tond*{#2}}

\newcommand{\expe}[1]{\mathbb{E}\quadr*{#1}}

\newcommand{\totvar}[2]{\mathrm{TV}\tond*{#1,  #2}}
\newcommand{\kldiv}[2]{{H}\tond*{#1\,\middle |\,  #2}}

\newcommand{\chisq}[2]{\chi^2\tond*{#1\,\middle |\,  #2}}

\newcommand{\pp}{\mathcal{P}}

\newcommand{\relden}[2]{\frac{\mathrm{d}#1}{\mathrm{d}#2}}

\newcommand{\pib}{\boldsymbol{\pi}}

\newcommand{\cpi}[1]{C_{\rm{P}}\tond*{#1}}
\newcommand{\hatcpi}[1]{\widehat{C}_{\rm{P}}\tond*{#1}}

\title{A transport approach to the cutoff phenomenon}

\author{Francesco Pedrotti and Justin Salez}

\begin{document}
	
	\maketitle
	
	\begin{abstract}
Substantial progress has recently been made in the understanding of the cutoff phenomenon for Markov processes, using an information-theoretic statistics known as  \emph{varentropy}  \cite{sal-2023,sal-2024,sal-2025,ped-sal-2025}.
		In the present paper, we propose an alternative approach which  bypasses the use of varentropy and   exploits instead a new W-TV transport inequality, combined with a classical  parabolic regularization estimate \cite{bob-gen-led-2001, ott-vil-2001}. While currently restricted to non-negatively curved processes on smooth spaces, our argument no longer requires the chain rule, nor any approximate version thereof. As applications, we recover the main result of \cite{sal-2025} establishing cutoff for the log-concave Langevin dynamics, and extend the conclusion to a widely-used discrete-time sampling algorithm known as the Proximal Sampler. 
	\end{abstract}

	\tableofcontents

	\section{Introduction}
The broad aim of the present work is to investigate the nature of the transition from out of equilibrium to equilibrium in MCMC algorithms, a popular class of methods for sampling from a target measure $\pi\in\pp\tond*{\R^d}$ by simulating a stochastic process that is ergodic with respect to $\pi$. We will here focus on two particular implementations that are widely used in practice: the Langevin dynamics, and the Proximal Sampler. Throughout the paper, we will assume that the target measure $\pi$ is \emph{log-concave}, i.e. of the form
	\begin{eqnarray*}
	\pi({\mathrm d}x) & = &  e^{-V(x)}{\mathrm d}x,
	\end{eqnarray*}
	 for some convex potential $V\in C^2(\R^d)$. 
 
	\subsection{The Langevin dynamics}The \emph{Langevin dynamics} with target distribution $\pi$ and initialization $\mu_0\in \pp\tond*{\R^d}$ is simply the solution to the stochastic differential equation
	\begin{eqnarray}
		 \label{eq:LD}
		X_0\sim \mu_0, \quad {\mathrm d}X_t & = & -\nabla V(X_t)\,{\mathrm d}t + \sqrt{2}\,{\mathrm d}B_t,
\end{eqnarray}
	where $(B_t)_{t\ge 0}$ is a standard $d-$dimensional Brownian motion. Under our assumptions, it is well  known that the marginal law $\mu_t\coloneqq\law(X_t)$ approaches  $\pi$ as $t\to\infty$. Moreover, a variety of  quantitative convergence guarantees are available in different metrics; see, e.g., \cite{bak-gen-led-2014} for an extensive treatment. In particular,  the    parabolic regularization estimate 
	\begin{eqnarray}\label{eq:KL-W2-reg}
		\forall t>0,\quad	\kldiv{\mu_t}{\pi} & \leq & \frac{W^2(\mu_0,\pi)}{4t},
	\end{eqnarray}
was proved in \cite{bob-gen-led-2001, ott-vil-2001}, and used to recover and generalize the celebrated HWI inequality from \cite{ott-vil-2000}. Here and throughout the paper, we use the classical notation  
	\begin{eqnarray*}
W^2(\mu,\pi) & \coloneqq &  \inf_{X\sim\mu, Y\sim \pi} \expe{\abs*{X-Y}^2},
\end{eqnarray*}
for the $2$-Wasserstein distance between $\mu$ and $\pi$, and 
\begin{eqnarray*}
\kldiv{\mu}{\pi}& \coloneqq & 
			\int \log\tond*{\relden{\mu}{\pi}} \mathrm{d}\mu,
	\end{eqnarray*}
for the relative entropy (or KL-divergence) of $\mu$ with respect to $\pi$. As usual, it is  understood that $\kldiv{\mu}{\pi}=\infty$ when $\mu$ is not absolutely continuous with respect to $\pi$. To translate \eqref{eq:KL-W2-reg} into the classical language of mixing times, we let
\begin{eqnarray}
\totvar{\mu}{\pi} & \coloneqq &  \sup_{A\in \mathcal{B}(\R^d)}\abs*{\mu(A)-\pi(A)}
 \end{eqnarray} 
denote the total-variation distance between $\mu$ and $\pi$, and we recall that the  \emph{mixing time}
 of the process $(X_t)_{t\ge 0}$ with initialization $\mu_0$  and precision $\varepsilon\in(0,1)$ is defined as
 \begin{eqnarray}
\label{def:mixing}
 \tmix(\mu_0,\varepsilon) & \coloneqq & \min\{t\ge 0\colon \totvar{\mu_t}{\pi}\le\varepsilon \}.
 \end{eqnarray} 
 It then readily follows from  \eqref{eq:KL-W2-reg}  and Pinsker's  inequality 
that 
\begin{eqnarray}
\label{eq:tmixLS}
\forall \varepsilon\in(0,1),\quad \tmix(\mu_0,\varepsilon) & \le & \frac{W^2(\mu_0,\pi)}{8\varepsilon^2}.
 \end{eqnarray} 
In concrete words, running the Langevin dynamics  for a time of order $W^2(\mu_0,\pi)$ suffices  to  approximately sample from $\pi$. Unfortunately, the continuous-time nature of the  Langevin dynamics is not appropriate for practical implementation, and a suitable  discretization   is required. The simplest choice is  the Euler--Maruyama method, in which a step size $h>0$ is chosen and a discrete-time process $(X_k)_{k\in\mathbb N}$ is produced via the stochastic recursion
	\begin{eqnarray}\label{eq:LMC}
		X_0\sim\mu_0,\quad X_{k+1} & = & X_k -h\nabla V(X_k) + \sqrt{2}\tond*{B_{(k+1)h} - B_{kh}}.
	\end{eqnarray}
	This sampling scheme is known as the \emph{Langevin Monte Carlo (LMC) algorithm} or  \emph{Unadjusted Langevin algorithm}. It is of fundamental importance, and has been extensively studied. 	We refer the unfamiliar reader to the book in progress \cite{che-book-2024+} and the references therein.  One drawback of this approach, however, is that the LMC algorithm is \emph{biased}:  the stationary distribution of \eqref{eq:LMC} is in general different from $\pi$. As a consequence, theoretical performance guarantees require the step size to be very small, resulting in an increased number of iterations compared to the theoretical time-scale \eqref{eq:tmixLS}.

	\subsection{The Proximal Sampler}
The \emph{Proximal Sampler}	is an unbiased discrete-time algorithm for sampling from $\pi$  introduced in \cite{lee-she-tia-2021}. 	We  refer the reader to \cite[Chapter 8]{che-book-2024+} or the recent papers \cite{che-eld-2022, mit-wib-2025, wib-2025} for more details. As usual, we denote by $\gamma_{x,t}$ the $d$-dimensional Gaussian law with mean $x$ and covariance matrix $tI_d$, and we use the shorthand $\gamma_t \coloneqq \gamma_{0,t}$. Given a step size $h>0$, we  consider a pair  $(X,Y)$ of $\R^d-$valued random variables with joint law
	\begin{eqnarray}
		\pib({\mathrm d}x\, {\mathrm d}y) & \propto &  \exp\tond*{-V(x) - \frac{\abs*{x-y}^2}{2h}}{\mathrm d}x\, {\mathrm d}y.
	\end{eqnarray}
The Proximal Sampler with target $\pi$ consists in applying alternating Gibbs sampling to $\pib$. More precisely,  a sequence  $X_0,Y_0,X_1,Y_1,\ldots$ of $\R^d$-valued random variables is produced by first sampling $X_0$ according to some prescribed initialization $\mu_0\in \pp\tond*{\R^d}$ and then inductively, for each $k\ge 0$, sampling $Y_k$ and $X_{k+1}$ according to the following laws:
\begin{enumerate}
\item Forward step: conditionally on $(X_0,Y_0,\ldots,X_{k-1},Y_{k-1},X_k)$,  $Y_k$ is distributed as 
\begin{eqnarray}
 \label{eq:forward-step-prox}
Y_k & \sim & \law\tond*{ Y \mid X = X_k}  \ = \ \gamma_{X_k, h}.
\end{eqnarray}
\item Backward step: conditionally on $(X_0,Y_0,\ldots,X_k,Y_k)$, $X_{k+1}$ is distributed as
\begin{eqnarray}
\label{eq:backward-step-prox}
X_{k+1} & \sim &  \law\tond*{ X \mid Y = Y_k} \ \propto \ \exp \tond*{-V(x)-\frac{\abs*{x-Y_k}^2}{2h}}\mathrm{d}x.
\end{eqnarray}
\end{enumerate}
Since the first marginal of $\pib$ is $\pi$, the algorithm is \emph{unbiased}, meaning that $\pi$ is stationary for the Markov chain $(X_k)_{k\ge 0}$. Moreover, the forward step is trivial to implement, as it   amounts to  adding an independent Gaussian noise. When   $h$ is small enough, the backward step  is also tractable, due to the regularizing effect of the quadratic potential in \eqref{eq:backward-step-prox}. For example, if $\pi$ is $\alpha$-log-concave and $\beta$-log-smooth (i.e. $\alpha I_d\preccurlyeq \hess V \preccurlyeq \beta I_d$), then the conditional law \eqref{eq:backward-step-prox} is $\tond*{\alpha + \frac1h}$-log-concave with condition number $\kappa_h = \frac{1+\beta h}{1+\alpha h} < \frac{\beta}{\alpha}$. As a result, several methods are available to efficiently generate  $X_{k+1}$, such as rejection sampling \cite{che-che-sal-wib-2022}, approximate rejection sampling \cite{fan-yua-che-2023}, or high-accuracy samplers \cite{alt-che-2024-faster}. Following a standard convention in this setting \cite{che-che-sal-wib-2022}, we will here assume to have access to a \emph{restricted Gaussian oracle} that implements the backward step \eqref{eq:backward-step-prox} exactly, and focus on the number of iterations needed for $\mu_k\coloneqq\law(X_k)$ to approach $\pi$, in the sense of  \eqref{def:mixing}.

	Although this is  not obvious from the above  description, the Proximal Sampler can be interpreted as a discretization of the Langevin dynamics \eqref{eq:LD}. This is particularly clear once we view those dynamics as minimizing schemes for the relative entropy functional  $\kldiv{\cdot}{\pi}$ in the Wasserstein space $\tond*{\pp_2(\R^d), W}$ \cite{jor-kin-ott-1998, che-book-2024+}.
	In light of this, it is not too surprising that many  classical convergence guarantees for the Langevin dynamics translate to the Proximal Sampler. In particular, the following analogue of the parabolic regularization estimate \eqref{eq:KL-W2-reg}  was recently established in \cite{che-che-sal-wib-2022}: 
\begin{eqnarray}
 \label{eq:KL-W2-reg-prox-sampl}
\forall k\in\mathbb N,\quad \kldiv{\mu_k}{\pi} & \leq & \frac{W^2\tond*{\mu_0,\pi}}{kh}.
	\end{eqnarray}	
By virtue of Pinsker's inequality, this  readily yields the (discrete-time) mixing-time estimate
	\begin{eqnarray}
\forall \varepsilon\in(0,1),\quad \tmix(\mu_0,\varepsilon) & \le & \left\lceil\frac{W^2(\mu_0,\pi)}{2h\varepsilon^2}\right\rceil.
 \end{eqnarray} 
In concrete words, running the Proximal Sampler with step size $h$ for roughly $W^2\tond*{\mu_0,\pi}/h$ iterations suffices to produce  approximate samples from the target distribution $\pi$.

\subsection{Main results}
Rather than asking \emph{how long} the Langevin dynamics or the Proximal Sampler should be run in order to be close to equilibrium, we would here like to understand \emph{how abrupt} their transition from out of equilibrium to equilibrium is. In other words, we seek to estimate the \emph{width} of the mixing window, defined for any precision $\varepsilon\in\tond*{0,\frac 12}$ by
\begin{eqnarray}
\wmix(\mu_0,\eps) & \coloneqq & \tmix(\mu_0,\varepsilon)-\tmix(\mu_0,1-\eps).
\end{eqnarray}
The analysis of this fundamental  quantity is a challenging task which, until recently, had only been carried out in a handful of models. Over the past couple of years, a  systematic approach to this problem was   developed in the series of works \cite{sal-2023,sal-2024,sal-2025,ped-sal-2025}, using an  information-theoretic notion called \emph{varentropy}. 	In the present paper, we propose an alternative approach which completely bypasses the use of varentropy, and instead  exploits the parabolic regularization estimates \eqref{eq:KL-W2-reg} and \eqref{eq:KL-W2-reg-prox-sampl} directly, in conjunction with a new W-TV transport inequality. As a first application, we are able to recover exactly the main result of \cite{sal-2025}, which reads as follows. Recall that the \emph{Poincaré constant} of  $\pi\in\pp\tond*{\R^d}$, denoted  $\cpi{\pi}$, is the smallest number such that
	\begin{eqnarray}
	\label{def:PI}
		\variwr{\pi}{f} & \leq & \cpi{\pi} \int \abs*{\nabla f}^2 d\pi,
	\end{eqnarray}
 for all smooth functions $f\colon \R^d\to\R$. In particular, $\cpi{\delta_x}=0$, for any $x\in\R^d$.
\begin{theorem}[Mixing window of the Langevin dynamics]
				\label{thm:mix-window-LD} 
The Langevin dynamics  with  a log-concave target $\pi\in\pp\tond*{\R^d}$ and an arbitrary initialization $\mu_0\in\pp\tond*{\R^d}$  satisfies
		\begin{eqnarray*}
\wmix(\mu_0,\eps) & \leq & \frac{3}{\eps}\tond*{\cpi{\pi} + \sqrt{\cpi{\pi}\cpi{\mu_0}}+ \sqrt{\cpi{\pi}\tmix(\mu_0,1-\eps)}},
		\end{eqnarray*}
for  any precision $\eps \in \tond*{0,\frac12}$.
	\end{theorem}
The interest of this estimate lies in its relation to the \emph{cutoff phenomenon}, a remarkably universal -- but still largely unexplained -- phase transition from out of equilibrium to equilibrium  undergone by certain ergodic Markov processes in an appropriate limit. We refer the unfamiliar reader to the recent lecture notes \cite{salez2025modernaspectsmarkovchains} and the references therein for an up-to-date introduction to this fascinating question.

	\begin{corollary}[Cutoff for the Langevin dynamics] \label{cor:cutoff-LD}
	Consider the setup of Theorem \ref{thm:mix-window-LD}, but assume that the ambient dimension $d$, the target  $\pi$, and the initialization $\mu_0$ now depend on an implicit parameter $n\in\mathbb N$, in such a way that  
		\begin{eqnarray}\label{eq:prod-cond-cpi}
	 \frac{\tmix(\mu_0,\eps)}{\cpi{\pi}+\sqrt{\cpi{\pi}\cpi{\mu_0}}} & \xrightarrow[n\to\infty]{} & +\infty,
\end{eqnarray}		 
 for some fixed $\eps \in (0,1)$. Then, a cutoff occurs, in the sense that for every  $\varepsilon\in(0,1)$,
		\begin{eqnarray}\label{def:cutoff}
	 \frac{\tmix(\mu_0,1-\eps)}{\tmix(\mu_0,\eps)} & \xrightarrow[n\to\infty]{} & 1.
\end{eqnarray}		 
	\end{corollary}
	Note that, in the standard setup where the initialization is a Dirac mass, our cutoff criterion \eqref{eq:prod-cond-cpi} reduces to the natural condition $\frac{\tmix(\mu_0,\eps)}{\cpi{\pi}}\to\infty$, which is known as  the \emph{product condition} in the classical literature on Markov chains   (see, e.g.,  \cite{salez2025modernaspectsmarkovchains}). 
\begin{remark}[Manifolds]We have here chosen to work in the Euclidean space $\R^d$ in order to keep the presentation simple and accessible. However, a careful inspection will convince the interested reader that our proof of Theorem \ref{thm:mix-window-LD} carries over to the more general setting of non-negatively curved diffusions on smooth complete weighted Riemannian manifolds. 
\end{remark}
As a second -- and genuinely new -- application of our transport approach to cutoff, we extend the above results to the Proximal Sampler, thereby tightening its relation to the Langevin dynamics. To lighten the formulas, we introduce the quantity
\begin{eqnarray*}
 \hatcpi{\pi} & \coloneqq & 1+\frac{\cpi{\pi}}{h},
 \end{eqnarray*}
 which, as we will see, can be seen as the natural discrete-time analogue of $\cpi{\pi}$.

\begin{theorem}[Mixing window of the Proximal Sampler]
\label{thm:mix-window-proximal sampler}
The Proximal Sampler with  log-concave target $\pi$, arbitrary initialization $\mu_0\in\pp\tond*{\R^d}$  and step size $h>0$   satisfies
		\begin{eqnarray*}
\wmix(\mu_0,\eps) & \leq & \frac{6}{\eps}\tond*{\hatcpi{\pi} + \sqrt{\hatcpi{\pi}\hatcpi{\mu_0}}+ \sqrt{\hatcpi{\pi}\tmix(\mu_0,1-\eps)}},
		\end{eqnarray*}
for any precision $\eps \in \tond*{0,\frac12}$.
\end{theorem}
\begin{corollary}[Cutoff for the Proximal Sampler]\label{cor:cutoff-proximal-sampler}
	Consider the setup of Theorem \ref{thm:mix-window-proximal sampler}, but assume that the ambient dimension $d$, the target $\pi$, the initialization $\mu_0$,  and the step size $h$ now depend on an implicit parameter $n\in\mathbb N$, in such a way that 
	\begin{eqnarray*} \label{eq:modified-prod-cond-cpi}
		\frac{\tmix(\mu_0,\varepsilon)}{\hatcpi{\pi}+\sqrt{\hatcpi{\pi}\hatcpi{\mu_0}}} & \xrightarrow[n\to\infty]{} &  \infty,
	\end{eqnarray*}
for some fixed $\eps \in (0,1)$. Then a cutoff occurs, in the sense of \eqref{def:cutoff} above.
\end{corollary}
Here again, the condition \eqref{eq:prod-cond-cpi} reduces to $ \frac{\tmix(\mu_0, \eps)}{\hatcpi{\pi}}\to\infty$  for  Dirac initializations. To the best of our knowledge, this is the very first result establishing cutoff for the Proximal Sampler. We emphasize that the latter is a discrete-time Markov process on a continuous state space, an object to which  the varentropy approach developed in \cite{sal-2023,sal-2024,sal-2025,ped-sal-2025,salez2025modernaspectsmarkovchains} does \emph{not} currently apply. Indeed, in that series of work,  varentropy is controlled using either the celebrated \emph{chain rule}, which notoriously fails in discrete time, or an approximate version of it involving a certain \emph{sparsity parameter}, which only makes sense on discrete   spaces. To bypass this limitation, our main idea is to replace the reverse Pinsker inequality \cite[Lemma 8]{sal-2023} where varentropy appears with the following 
W-TV transport inequality, which seems new and of independent interest. 
	
 	\begin{theorem}[W-TV transport inequality]\label{th:WTV}For any $\mu,\nu\in\pp\tond*{\R^d}$,
	\begin{eqnarray*}
	W^2(\mu,\nu)&  \le &\frac{ 4\left(\cpi{\mu}+\cpi{\nu}\right)\totvar{\mu}{\nu}}{1-\totvar{\mu}{\nu}}.
	\end{eqnarray*}
	\end{theorem}

	\section{Proofs}
\subsection{The W-TV transport inequality}

In this section, we prove Theorem \ref{th:WTV}. Given two probability measures $\mu,\nu \in \pp \tond*{\R^d}$, we recall that the chi-squared divergence of $\nu$ w.r.t. $\mu$ is defined by the formula
\begin{eqnarray*}
\chisq{\nu}{\mu} & \coloneqq & \int \tond*{\relden{\nu}{\mu}-1} \mathrm{d}\nu,
\end{eqnarray*}
with  $\chisq{\nu}{\mu}=\infty$ if $\nu$ is not absolutely continuous w.r.t. $\mu$. Our starting point is the following  transport-variance inequality, whose proof can be found in \cite{liu-2020}.
\begin{lemma}[Transport-variance inequality] \label{lm:Wchisq}For any $\mu,\nu \in \pp \tond*{\R^d}$,
	\begin{eqnarray*}
			W^2(\mu,\nu) & \leq & 2\cpi{\mu} \chisq{\nu}{\mu}.
	\end{eqnarray*}
\end{lemma}
Unfortunately, the chi-squared divergence appearing here could be arbitrarily large compared to the total-variation term with which we seek to control $W^2(\mu,\nu)$. To preclude such pathologies, we introduce a  probability measure $\lambda$ which  \emph{interpolates nicely} between  $\mu$ and $\nu$, in the sense of having small Radon-Nikodym derivatives w.r.t. both. 
\begin{lemma}[Interpolation]\label{lm:interp} Fix two probability measures $\mu,\nu$ on a measurable space, with $\totvar{\mu}{\nu}<1$. Then, there exists a  probability measure $\lambda$ which is absolutely continuous with respect to both $\mu$ and $\nu$, such that
	\begin{eqnarray*}
		\left\|\relden{\lambda}{\mu}\right\|_\infty\vee\ \ \left\|\relden{\lambda}{\nu}\right\|_\infty	 & \leq & \frac{1}{1-\totvar{\mu}{\nu}}.
	\end{eqnarray*}
\end{lemma}
\begin{proof}
Fix a measure $\sigma$ with respect to which both  $\mu$ and $\nu$ are absolutely continuous (e.g., $\sigma:=\mu+\nu$), and let $f:=\relden{\mu}{\sigma}$ and $g:=\relden{\nu}{\sigma}$ denote the corresponding Radon-Nikodym derivatives. With this notation at hand, we classically have the integral representation
\begin{eqnarray*}
\totvar{\mu}{\nu} & = & 1-\int \tond*{f \wedge g}\, {\mathrm d}\sigma.
\end{eqnarray*}
Consequently, we can define a probability measure $\lambda$ by the formula
\begin{eqnarray*}
{\mathrm d}\lambda & := & \frac{f\wedge g}{1-\totvar{\mu}{\nu}}\,\mathrm{d}{\sigma}.
\end{eqnarray*}
This measure is absolutely continuous w.r.t. both $\mu$ and $\nu$, because $\lambda(A)\le \frac{\mu(A)\wedge \nu(A)}{1-\totvar{\mu}{\nu}},$ for every measurable set $A$. Moreover, the corresponding Radon-Nikodym derivatives are
\begin{eqnarray*}
\relden{\lambda}{\mu} \ = \ \frac{1\wedge \frac g f}{1-\totvar{\mu}{\nu}},& \textrm{ and } & 
\relden{\lambda}{\nu} \ = \ \frac{1\wedge \frac f g}{1-\totvar{\mu}{\nu}},
\end{eqnarray*}
those formulae being interpreted as zero outside  $\supp(\lambda):=\{f\wedge g>0\}$. 
\end{proof}
We now have everything we need to establish Theorem \ref{th:WTV}. 
\begin{proof}[Proof of Theorem \ref{th:WTV}]Fix $\mu,\nu\in\pp \tond*{\R^d}$ 
 and let $\lambda$ be as in Lemma \ref{lm:interp}. Then,
\begin{eqnarray*}
W^2(\mu,\nu) & \le & 2W^2(\mu,\lambda)+2W^2(\nu,\lambda)\\
& \le & 4\cpi{\mu}\chisq{\lambda}{\mu}+4\cpi{\nu}\chisq{\lambda}{\nu},
\end{eqnarray*}
by the triangle inequality and Lemma \ref{lm:Wchisq}. On the other hand, we have the crude bound
\begin{eqnarray*}
\chisq{\lambda}{\mu} & \le & \left\|\relden{\lambda}{\mu}\right\|_\infty-1 
\ \le \ \frac{\totvar{\mu}{\nu}}{1-\totvar{\mu}{\nu}},
\end{eqnarray*}
by Lemma \ref{lm:interp}, and similarly with $\nu$ instead of $\mu$. 
\end{proof}

	\subsection{Cutoff for the Langevin dynamics}

	In this section, we prove  Theorem \ref{thm:mix-window-LD} and Corollary \ref{cor:cutoff-LD}. Consider the Langevin dynamics  \eqref{eq:LD} with  target  $\pi$ and  initialization $\mu_0\in\pp\tond*{\R^d}$, and write $\mu_t = \law(X_t)$ for the  law at time $t\ge 0$. As is well known from the Bakry--\'Emery theory \cite{bak-gen-led-2014} (see Remark \ref{rk:extensions} below for an alternative proof), the log-concavity of $\pi$ ensures the local Poincaré inequality
\begin{eqnarray}
\label{eq:LPI}
\forall t\ge 0,\quad	\cpi{\mu_t} & \le & \cpi{\mu_0}+2t.
	\end{eqnarray}
Another ingredient that we will need is the basic mixing-time estimate
\begin{eqnarray}
\label{eq:sal-fast-tv-mixing-small-kl}
\tmixa{\mu_0,\eps} & \leq & \frac{\cpi{\pi}\tond*{1+\kldiv{\mu_0}{\pi}}}{\eps},
\end{eqnarray}
borrowed from \cite[Lemma 7]{sal-2023}, and which relies on the classical fact that
	\begin{eqnarray*}
\forall t\ge 0,\quad 	\chisq{\mu_t}{\pi}  & \le & e^{-2\cpi{\pi}t}	\chisq{\mu_0}{\pi},
	\end{eqnarray*}
together with an easy interpolation argument between $\chisq{\mu_t}{\pi},\kldiv{\mu_t}{\pi}$ and $\totvar{\mu_t}{\pi}$. 
	\begin{proof}[Proof of Theorem \ref{thm:mix-window-LD}]
Fix  $\eps\in \tond*{0,\frac12}$ and set $t_0 := \tmix(\mu_0,1-\eps)$. By the very definition of $t_0$, our W-TV transport inequality (Theorem \ref{th:WTV}) gives
\begin{eqnarray*}
W^2(\mu_{t_0},\pi) & \leq & \frac{4\cpi{\pi}+4\cpi{\mu_{t_0}}}{\varepsilon}.
\end{eqnarray*}
Therefore, the parabolic regularization estimate  \eqref{eq:KL-W2-reg} applied to $\mu_{t_0}$ instead of $\mu_0$ yields
\begin{eqnarray*}
\forall s\ge 0,\quad \kldiv{\mu_{t_0+s}}{\pi} & \leq & \frac{\cpi{\pi}+\cpi{\mu_{t_0}}}{s \eps}.
\end{eqnarray*}
On the other hand, applying \eqref{eq:sal-fast-tv-mixing-small-kl} to $\mu_t$ instead of $\mu_0$ ensures that
\begin{eqnarray*}
\tmix(\mu_0,\varepsilon) & \leq & t+ \frac{\cpi{\pi}\tond*{1+\kldiv{\mu_t}{\pi}}}{\eps},
\end{eqnarray*}
 for any $t\ge 0$. Choosing $t=t_0+s$ and combining this with the previous line, we obtain
\begin{eqnarray*}
\wmix(\mu_0,\varepsilon) & \le & s+\frac{\cpi{\pi}}{\eps}+\frac{C_{\mathrm P}^2(\pi)+\cpi{\pi}\cpi{\mu_{t_0}}}{s \eps^2}.
\end{eqnarray*}
Since this bound is valid for any $s\ge 0$, we may finally optimize on $s$ to conclude that
\begin{eqnarray*}
\wmix(\mu_0,\varepsilon) & \le & \frac{\cpi{\pi}}{\eps}+\frac{2}{\eps}\sqrt{C_{\mathrm P}^2(\pi)+\cpi{\pi}\cpi{\mu_{t_0}}}.
\end{eqnarray*}
This implies the desired estimate, thanks to \eqref{eq:LPI} and  the subadditivity of $\sqrt{\cdot}$. 
	\end{proof}
		\begin{proof}[Proof of Corollary \ref{cor:cutoff-proximal-sampler}]
We now let the ambient dimension $d$, the target $\pi$ and the initialization $\mu_0$ depend on an implicit parameter $n\in \mathbb N$, in such a way that the condition (\ref{eq:prod-cond-cpi}) holds for some $\varepsilon\in(0,1)$. Since $\tmix(\mu_0,\varepsilon)$ is a non-increasing function of $\varepsilon$,  (\ref{eq:prod-cond-cpi}) must in fact holds for every small enough $\varepsilon>0$, and Theorem \ref{thm:mix-window-LD} then readily implies that 
	\begin{eqnarray}
 \frac{\wmix(\mu_0,\varepsilon)}{\tmix(\mu_0,\varepsilon)} & \xrightarrow[n\to\infty]{} & 0. 
	\end{eqnarray}
Since this holds for every small enough $\varepsilon>0$, the cutoff phenomenon (\ref{def:cutoff}) follows. 
		\end{proof}
	\subsection{Cutoff for the Proximal Sampler}
To prove Theorem \ref{thm:mix-window-proximal sampler}, we  fix a log-concave target $\pi\in\pp\tond*{\R^d}$, a step size $h>0$, and an initialization $\mu_0\in\pp\tond*{\R^d}$, and we consider the random sequence  $X_0,Y_0,X_1,Y_1,\ldots$  generated by the Proximal Sampler \eqref{eq:forward-step-prox}-\eqref{eq:backward-step-prox}. We write $\mu_k=\mathrm{law}(X_k)$. The main ingredient we need in order to mimic the proof of Theorem \ref{thm:mix-window-LD} is a version of the local Poincaré inequality \eqref{eq:LPI} for the Proximal Sampler, provided in the following lemma.
	\begin{lemma}[local Poincaré inequality for the Proximal Sampler]\label{lm:local-concentration-PI}We have
	\begin{eqnarray*}
\forall k\in\mathbb N,\quad \cpi{\mu_k} & \le & \cpi{\mu_0} + 2kh.
		\end{eqnarray*}
	\end{lemma}
	\begin{proof}
Let us use the convenient short-hand $\cpi{U}:=\cpi{\mathrm{law}(U)}$ when $U$ is a $\R^d$-valued random variable. By induction, it is enough to prove the claim when $k=1$, i.e.
\begin{eqnarray*}
\cpi{X_1} & = & \cpi{X_0}+2h.
\end{eqnarray*}
First observe that, by construction, the random variable $Y_0-X_0$ is $\gamma_h$-distributed and independent of $X_0$. Using the sub-additivity of the Poincaré constant under convolutions and the  Gaussian Poincaré inequality (see, e.g., \cite{bak-gen-led-2014}), we  deduce that 
\begin{eqnarray*}
\cpi{Y_0} & \le & \cpi{X_0}+h,
\end{eqnarray*}
which reduces our task to proving that
\begin{eqnarray}
\label{goal}
\cpi{X_1} & \le  &\cpi{Y_0}+h.
\end{eqnarray}
Let us first establish this under the additional assumption that $\pi$ is log-smooth, i.e.  $\hess V \preccurlyeq \beta I_d$ for some $\beta<\infty$. 
To do so, we rely on a clever continuous-time stochastic interpolation between $Y_0$ and $X_1$ introduced in \cite{che-che-sal-wib-2022}.
		More precisely, it is shown therein  that $X_1\stackrel{d}{=}U_h$, where $(U_t)_{t\in[0,h]}$ solves the SDE
		\begin{eqnarray}
		\label{eq:backward-BM}
		U_0 = Y_0, \quad \mathrm{d}U_t & = & \nabla \log  f_{h-t}(U_t)\mathrm{d}t + \mathrm{d}B_t,
		\end{eqnarray}
with $f_t$ denoting the density of $\pi*\gamma_t$. Following a strategy used in \cite{vem-wib-2019}, we now track the evolution of the Poincaré constant along an appropriate time-discretization of this SDE. Specifically, given a resolution $n\in\mathbb N$, we consider the Euler--Maruyama discretization $(\tilde U_0,\ldots,\tilde U_n)$ of  \eqref{eq:backward-BM} with  step size $\delta\coloneqq \frac{h}{n}$, defined inductively by
		\begin{eqnarray*}
\tilde U_0 = Y_0, \quad  \tilde U_{j+1} & \coloneqq & \tilde U_j + \delta \nabla \log f_{h-\delta j}(\tilde U_j) + B_{\delta\tond{j+1}} -B_{\delta j}.
		\end{eqnarray*}
As above, the sub-additivity of $\mu\mapsto\cpi{\mu}$ under convolutions  yields
		\begin{eqnarray}
		\label{eq:hessft}
\cpi{\tilde U_{j+1}} & \leq & \cpi{\tilde U_j + \delta \nabla \log f_{h-\delta j}(\tilde U_j)} + \delta. 
		\end{eqnarray}
To estimate the right-hand side, we recall that by assumption, $0\preccurlyeq -\nabla^2 \log f_0\preccurlyeq \beta I_d$ for some $\beta<\infty$, and that this property is preserved under the heat flow, i.e.
		\begin{eqnarray*}
		\forall t\ge 0,\quad 0 \ \preccurlyeq\ -\hess \log f_t & \preccurlyeq & \beta I_d,
		\end{eqnarray*}
see \cite{sau-wel-2014} for the lower bound and equation (6) in \cite{mik-she-2023} for the upper bound.
Consequently,  the gradient-descent map $x \to x + \delta \nabla \log f_{t}(x)$ is $1$-Lipschitz as soon as $\beta\delta\le 1$, which we can enforce by choosing $n\ge h\beta$. Since the Poincaré constant can not increase under $1-$Lipschitz pushforwards (see \cite{cor_era-2002}), we deduce that 
\begin{eqnarray*}
\cpi{\tilde U_j + \delta \nabla \log f_{h-\delta j}(\tilde U_j)} & \le & \cpi{\tilde U_j}.
\end{eqnarray*}
Inserting this into \eqref{eq:hessft} and solving the resulting recursion, we conclude that 
\begin{eqnarray*}
\cpi{\tilde U_n} & \le & \cpi{Y_0}+h.
\end{eqnarray*}
Sending $n\to\infty$ 
gives \eqref{goal}, since the Euler--Maruyama approximation $\tilde U_n$ converges in distribution to $U_h\stackrel{d}{=}X_1$ as the resolution $n$ tends to infinity. Finally, to remove our log-smoothness assumption on $\pi$, we fix a regularization parameter $\eps>0$ and consider the random sequence  $X_0^\eps,Y_0^\eps,X_1^\eps,\ldots$  generated by the Proximal Sampler with  initialization $\mu_0$, step size $h$, and regularized target $\pi_\eps:=\pi*\gamma_\varepsilon$. Since the latter is log-concave and log-smooth, the first step of the proof ensures that 
\begin{eqnarray}
\label{goal:eps}
\cpi{X^\eps_1} & \le & \cpi{Y_0^\eps}+h. 
\end{eqnarray}
But by construction, we have $\mathrm{law}(Y_0^\eps)=\mu_0*\gamma_h=\mathrm{law}(Y_0)$, and for each $y\in\R^d$, 
\begin{eqnarray*}
\mathrm{law}(X^\eps_1|Y^\eps_0=y) & = & \frac{e^{-\frac{|x-y|^2}{2h}}\pi_\eps(\mathrm{d}x)}{\int e^{-\frac{|z-y|^2}{2h}}\pi_\eps(\mathrm{d}z)} \ \xrightarrow[\varepsilon\to 0]{} \ 
 \frac{e^{-\frac{|x-y|^2}{2h}}\pi(\mathrm{d}x)}{\int e^{-\frac{|z-y|^2}{2h}}\pi(\mathrm{d}z)} 
 \ = \ \mathrm{law}(X_1|Y_0=y),
\end{eqnarray*}
simply because $\pi_\eps\to\pi$ as $\eps\to 0$. Thus, $\mathrm{law}(X_1^\eps)\to \mathrm{law}(X_1)$ as $\eps\to 0$, and we may safely pass to the limit in \eqref{goal:eps} to obtain \eqref{goal}.
\end{proof}

	\begin{remark}[Extensions] \label{rk:extensions}The above argument is rather robust. For example, applying it to \eqref{eq:LMC} gives a simple alternative proof of the celebrated  local Poincaré inequality \eqref{eq:LPI}, and the same reasoning actually also yields local log-Sobolev inequalities. When the potential $V$ is strongly log-concave, sharp improved estimates on those constants can be derived accordingly, using the strong contractivity of the gradient-descent map.
	\end{remark}
We will also need the following analogue of the mixing-time estimate \eqref{eq:sal-fast-tv-mixing-small-kl}. 
	\begin{lemma}[Mixing-time estimate for the Proximal Sampler]\label{lem:fast-tv-mix-small-kl-proximal}
		We have
		\begin{eqnarray*}
\tmixa{\mu_0,\eps} & \leq & \ceil*{\hatcpi{\pi}\frac{1+\kldiv{\mu_0}{\pi}}{\varepsilon}}.
		\end{eqnarray*}
	\end{lemma}
\begin{proof}
It was shown in \cite{che-che-sal-wib-2022} that for any $k\in\mathbb N$, 
\begin{eqnarray*}
\chisq{\mu_k}{\pi} & \leq & \tond*{1+\frac{h}{\cpi{\pi}}}^{-2k}\chisq{\mu_0}{\pi}\\
& \leq & \exp\left(-\frac{2k}{\hatcpi{\pi}}\right)\chisq{\mu_0}{\pi},
	\end{eqnarray*}   
where the second line follows from our definition of $\hatcpi{\pi}$ and the bound $e^{\frac{1}{u}}\le \frac{u}{u-1}$, valid for any $u\ge 1$. The remainder of the proof is then exactly as in \cite[Lemma 7]{sal-2023}. 
\end{proof}
We now have everything we need to mimic the proof of Theorem \ref{thm:mix-window-LD}.
		\begin{proof}[Proof of Theorem \ref{thm:mix-window-proximal sampler}]
	Fix $\eps\in \tond*{0,\frac12}$ and	set $k_0:= \tmixa{\mu_0,1-\eps}$. Our W-TV transport inequality (Theorem \ref{th:WTV}) combined with the parabolic regularization estimate \eqref{eq:KL-W2-reg-prox-sampl} gives
\begin{eqnarray*}
\kldiv{\mu_{k_0+k}}{\pi} & \leq & \frac{4\hatcpi{\pi}+4\hatcpi{\mu_{k_0}}}{\eps k}.
\end{eqnarray*}
As above, we can then apply Lemma \ref{lem:fast-tv-mix-small-kl-proximal} to $\mu_{k_0+k}$ instead of $\mu_0$ to obtain
\begin{eqnarray*}
\wmix(\mu_0,\eps) & \leq & k + 1+ \hatcpi{\pi}\tond*{\frac{1}{\eps}+\frac{4\hatcpi{\pi}+4\hatcpi{\mu_{k_0}}}{\eps^2 k}}.
\end{eqnarray*}
But this holds for any $k\in\mathbb N$, and choosing $k=\ceil*{\frac{2}{\varepsilon}\sqrt{\hatcpi{\pi}\tond*{\hatcpi{\pi}+\hatcpi{\mu_{k_0}}}}}$ yields
\begin{eqnarray*}
\wmix(\mu_0,\eps) & \leq & 2+\frac{\hatcpi{\pi}}{\varepsilon}+ \frac{4}{\varepsilon}\sqrt{\hatcpi{\pi}\tond*{\hatcpi{\pi}+\hatcpi{\mu_{k_0}}}}.
\end{eqnarray*}
The result now readily follows from Lemma \ref{lm:local-concentration-PI} and the sub-additivity of $\sqrt{\cdot}$.		
	\end{proof}
		\subsection*{Acknowledgment}
F.P. thanks Yuansi Chen for helpful comments.
    J.S. is supported by the ERC consolidator grant CUTOFF (101123174). Views and opinions expressed are however those of the authors only and do not
necessarily reflect those of the European Union or the European Research Council Executive
Agency. Neither the European Union nor the granting authority can be held responsible.

\printbibliography

\end{document}